\documentclass[11pt]{amsart}

\newtheorem{theorem}{Theorem}[section]

\newtheorem{corollary}[theorem]{Corollary}

\newtheorem{definition}[theorem]{Definition}
\newtheorem{lemma}[theorem]{Lemma}

\newtheorem{proposition}[theorem]{Proposition}

\newtheorem{example}[theorem]{Example}

\usepackage[colorlinks, bookmarks=true]{hyperref}
\usepackage{color,graphicx,shortvrb}
\usepackage[active]{srcltx} 

\def\J#1#2#3{ \left\{ #1,#2,#3 \right\} }

\def\11{\textbf{$1$}}

\usepackage{enumerate}
\usepackage{amsmath}
\usepackage{amssymb}
\usepackage[latin 1]{inputenc}

\title[Pointwise-generalized-inverses of linear maps]{Pointwise-generalized-inverses of linear maps between C$^*$-algebras and JB$^*$-triples}

\author[Ben Ali Essaleh]{Ahlem Ben Ali Essaleh}
\address{Faculte des Sciences de Monastir, Département de Mathématiques, Avenue de L'environnement, 5019 Monastir, Tunisia}
\email{ahlem.benalisaleh@gmail.com}

\author[Peralta]{Antonio M. Peralta}
\address{Departamento de An{\'a}lisis Matem{\'a}tico, Facultad de
Ciencias, Universidad de Granada, 18071 Granada, Spain.}
\email{aperalta@ugr.es}

\author[Ram\'{i}rez]{Mar{\'\i}a Isabel Ram{\'\i}rez}
\address{Departamento Matem\'aticas, Universidad de
Almer\'ia, 04120 Almer\'ia, Spain} \email{mramirez@ual.es}

\thanks{The first author was supported by the Higher Education And Scientific Research Ministry In Tunisia, UR11ES52~: Analyse, G{\'e}om{\'e}trie et Applications. The last two authors were supported by the Spanish Ministry of Economy and Competitiveness and European Regional Development Fund project no. MTM2014-58984-P and Junta de Andaluc\'{\i}a grant FQM375.}

\begin{document}

\begin{abstract} We study pointwise-generalized-inverses of linear maps between C$^*$-algebras. Let $\Phi$ and $\Psi$ be linear maps between complex Banach algebras $A$ and $B$. We say that $\Psi$ is a pointwise-generalized-inverse of $\Phi$ if $\Phi(aba)=\Phi(a)\Psi(b)\Phi(a),$ for every $a,b\in A$. The pair $(\Phi,\Psi)$ is Jordan-triple multiplicative if $\Phi$ is a pointwise-generalized-inverse of $\Psi$ and the latter is a pointwise-generalized-inverse of $\Phi$. We study the basic properties of this maps in connection with Jordan homomorphism, triple homomorphisms and strongly preservers. We also determine conditions to guarantee the automatic continuity of the pointwise-generalized-inverse of continuous operator between C$^*$-algebras. An appropriate generalization is introduced in the setting of JB$^*$-triples.
\end{abstract}

\maketitle

\section{Introduction}

Let $\Delta : A\to B$ be a mapping between two Banach algebras. Accordingly to the standard literature (see \cite{Mol06a,Mol06,Lu03} and \cite{KimPark}) we shall say that $\Delta$ is a \emph{Jordan triple map} (respectively, \emph{Jordan triple product homomorphism} or a \emph{Jordan triple multiplicative} mapping) if the identity $$\Delta (a b c + c b a  ) =\Delta (a) \Delta (b) \Delta (c) + \Delta (c) \Delta (b) \Delta (a) $$ (respectively, $\Delta (a b a) =\Delta (a) \Delta (b) \Delta (a)$) holds for every $a,b,c\in A$. For a linear map $T:A\to B,$ it is easy to see that $T$ is a Jordan triple map if, and only if, Jordan triple product homomorphism. In \cite{Mol06}, L. Molnar gives a complete description of those Jordan triple multiplicative bijections $\Phi$ between the self adjoint parts of two von Neumann algebras $M$ and $N$. F. Lu studies in \cite{Lu03} bijective maps from  a standard operator algebra into a $\mathbb{Q}$-algebra which are generalizations of Jordan triple multiplicative maps.\smallskip

In this paper we introduce a new point of view by introducing and studying pairs of linear maps which are Jordan triple multiplicative.  Henceforth let $A$ and $B$ denote two complex Banach algebras.

\begin{definition}\label{de pseudo inverses} Let $\Phi, \Psi :A\rightarrow B$ be linear maps.
We shall say that $\Psi$ is a pointwise-generalized-inverse {\rm(}pg-inverse for short{\rm)} of $\Phi$ if the identity
$$\Phi(aba)=\Phi(a)\Psi(b)\Phi(a),$$ holds for all $a,b\in A$.
If in addition $\Phi$ also is a pointwise-generalized-inverse of $\Psi,$ we shall say that $\Psi$ is a
normalized-pointwise-generalized-inverse {\rm(}normalized-pg-inverse for short{\rm)} of $\Phi.$ In this case, we shall simply say that $(\Phi,\Psi)$ is Jordan-triple multiplicative.
\end{definition}

Let us observe that, in the linear setting, $\Psi: A\to B$ is a pg-inverse of $\Phi$ if and only if $$\Phi(abc + cba )=\Phi(a)\Psi(b)\Phi(c) + \Phi(c)\Psi(b)\Phi(a),$$ for all $a,b,c\in A$.\smallskip


Every Jordan homomorphism (in particular, every homomorphism and every anti-homomorphism) $\pi: A\to B$ admits a pg-inverse. Actually, the couple $(\pi,\pi)$ is Jordan-triple multiplicative.\smallskip

An element $a$ in an associative ring $\mathcal{R}$ is called \emph{regular} or \emph{von Neumann regular} if it admits a \emph{generalized inverse} $b$ in $\mathcal{R}$ satisfying $a ba = a$. The element $b$ is not, in general, unique. Under these hypothesis $ab$ and $ba$ are idempotents with $(ab) a = a (ba) =a$.  If the identities $a ba =a$ and $bab =b$ hold we say that $b$ is a \emph{normalized generalized inverse} of $a$. An element $a$ may admit many different normalized generalized inverses. However, every regular element $a$ in a C$^*$-algebra $A$ admits a unique \emph{Moore-Penrose inverse} that is, a normalized generalized inverse $b$ such that $a b$ and $ b a$ are projections (i.e. self-adjoint idempotents) in $A$ (see \cite[Theorems 5 and 6]{HarMb92}). The unique Moore-Penrose inverse of a regular element $a$ will be denoted by $a^{\dag}$ for the unique Moore-Penrose inverse of $a$.\smallskip

A linear map between C$^*$-algebras admitting a pg-inverse is a \emph{weak preserver}, that is, maps regular elements to regular elements (see Lemma \ref{l basic pg inv}). However, we shall show later the existence of linear maps between C$^*$-algebras preserving regular elements but not admitting a pg-inverse (see Example \ref{example BurMarMoPe}). Being a linear weak preserver between C$^*$-algebras is not a completely determining condition, actually, for an infinite-dimensional complex separable Hilbert space $H$, a bijective continuous unital linear map preserving generalized invertibility in both directions $\Phi: B(H) \to B(H)$ leaves invariant the ideal of all compact operators, and the induced linear map on the Calkin algebra is either an automorphism or an antiautomorphism (see \cite{MbRodSem06}).\smallskip

In Proposition \ref{prop2} we show that a linear map $\Phi:A\rightarrow B$ between complex Banach algebras with $A$ unital, admits a normalized-pg-inverse if and only if one of the following statements holds:
\begin{enumerate}
    \item[$(b)$] There exists a Jordan homomorphism  $T:A\rightarrow B$ such that $\Phi= R_{\Phi(1)}\circ T$ and $\Phi(1) B=T(1) B;$
    \item[$(c)$] There exists a Jordan homomorphism $S:A\rightarrow B$ such that $\Phi=L_{\Phi(1)}\circ S$ and $B \Phi(1)=B S(1).$
\end{enumerate} A similar conclusion remains true for pair of bounded linear maps between general C$^*$-algebras which are Jordan-triple multiplicative (see Corollary \ref{c prop2 continuity instead unital}).\smallskip

A linear map $\Phi$ between C$^*$-algebras satisfying that $\Phi (a^{\dag}) = \Phi (a)^{\dag}$ for every regular element $a$ in the domain is called a strongly preserver. Strongly preservers between C$^*$-algebras and subsequent generalizations to JB$^*$-triples have been studied in \cite{BMM11,BurMarMor12,BurMarMoPe2016}. Following the conclusions of the above paragraph we can easily find a bounded linear map between C$^*$-algebras admitting a normalized-pg-inverse which is not a strongly preserver. In this setting, we shall show in Theorem \ref{thm2} that for each pair of linear maps between C$^*$-algebras $\Phi, \Psi:A\rightarrow B$ such that $(\Phi,\Psi)$ is Jordan-triple multiplicative, the following statements are equivalent:
\begin{enumerate}[$(a)$]
\item $\Phi$ and $\Psi$ are contractive;
\item  $\Psi(a)=\Phi(a^*)^*,$ for every $a\in A;$
\item $\Phi$ and $\Psi$ are triple homomorphisms.
\end{enumerate} When $A$ is unital the above conditions are equivalent to the following:\begin{enumerate}\item[$(d)$] $\Phi$ and $\Psi$ are strongly preservers,
\end{enumerate}  (see \cite[Theorem 3.5]{BurMarMor12}). As a consequence, we prove that every contractive Jordan homomorphism between C$^*$-algebras or between JB$^*$-algebras is a Jordan $^*$-homomorphism (cf. Corollaries \ref{c Jordan homomorphisms} and \ref{c contractive are symetric in JB*-algebras}).\smallskip

Let $\Phi,\Psi:A\rightarrow B$ be linear maps between complex Banach algebras. If $A$ is unital and $(\Phi,\Psi)$ is Jordan-triple multiplicative, then $\Phi$ is norm continuous if and only if $\Psi$ is norm continuous (cf. Lemma \ref{l basic pg inv}). In the non-unital setting this conclusion becomes a difficult question. In section \ref{sec 3}, we explore this problem by showing that if $\Phi, \Psi: c_0 \to c_0$ are linear maps such that $\Phi$ is continuous and $(\Phi,\Psi)$ is Jordan-triple multiplicative, then $\Psi$ is continuous (see Proposition \ref{p automatic continuity on c0}). In the non-commutative setting, we prove that if $\Phi, \Psi: K(H_1) \to K(H_2)$ are linear maps such that $\Phi$ is continuous and $(\Phi,\Psi)$ is Jordan-triple multiplicative, then $\Phi$ admits a continuous normalized-pg-inverse (see Theorem \ref{p automatic continuity on K(H)}).\smallskip

In the last section we extend the notion of being pg-invertible to the setting of JB$^*$-triples.

\subsection{Preliminaries and background}\label{subsec: preliminaries}

We gather some basic facts, definitions and references in this subsection. We recall that \emph{JB*-triple} is a complex Jordan triple system which is also a Banach space satisfying the following axioms:
\begin{enumerate}[$(a)$]
\item The map $L(x,x)$ is an hermitian operator with
non-negative spectrum for all $x\in { E}$.
\item $\| \J xxx \| = \|x\|^{3}$ for all $x\in { E}$.
\end{enumerate} where $L(x,y)(z):=\J xyz,$ for all $x,y,z $ in $E$ (see \cite{Ka83} and \cite{CaRod14}).\smallskip

The attractive of this definition relies, among other holomorphic properties, on the fact that
every C$^*$-algebra is a JB$^*$-triple with respect to $$\J xyz := 2^{-1} (x
y^* z + z y^* x),$$ the Banach space $B(H,K)$ of all bounded linear operators
between two complex Hilbert spaces $H,K$ is also an example of a JB$^*$-triple
with respect to the triple product given above, and every JB$^*$-algebra is a JB$^*$-triple with triple
product $$\J abc := (a \circ b^*) \circ c + (c\circ b^*) \circ a - (a\circ c)
\circ b^*.$$

In a clear analogy with von Neumann algebras, a JB$^*$-triple which is also a dual Banach space is called a \emph{JBW$^*$-triple}. Every JBW$^*$-triple admits a (unique) isometric predual and its triple product is separately weak$^*$ continuous \cite{BarTi}. The second dual of a JB$^*$-triple $E$ is a JBW$^*$-triple with a product extending the product of $E$ \cite{Di}.\smallskip

Projections are frequently applied to produce approximation and spectral resolutions of hermitian elements in von Neumann algebras. In the wider setting of JBW$^*$-triple this role is played by tripotents. We recall that an element $e$ in a JB$^*$-triple $E$ is called a \emph{tripotent} if $\J eee =e$. Each tripotent $e$ in $E$ produces a \emph{Peirce decomposition} of $E$ in the form
$$E= E_{2} (e) \oplus E_{1} (e) \oplus E_0 (e),$$ where for
$i=0,1,2,$ $E_i (e)$ is the $\frac{i}{2}$ eigenspace of $L(e,e)$
(compare \cite[\S 4.2.2]{CaRod14}). The projection of $E$ onto $E_i(e)$ is denoted by $P_i(e)$. 
\smallskip

It is known that the Peirce space $E_2 (e)$ is a JB$^*$-algebra with product
$x\circ_e y := \J xey$ and involution $x^{\sharp_e} := \J exe$.
\smallskip

For additional details on JB$^*$-algebras and JB$^*$-triples the reader is referred to the encyclopedic monograph \cite{CaRod14}.\smallskip

For the purposes of this paper, we also consider von Neumann regular elements in the wider setting of JB$^*$-triples (see subsection \ref{subsec: preliminaries} for the concrete definitions). Let $a$ be an element in a JB$^*$-triple $E$. Following the standard notation in \cite{FerGarSanSi}, \cite{Ka96} and \cite{BurKaMoPeRa08} we shall say that $a$ is \emph{von Neumann regular} if $a\in Q(a) (E) =\{a, E, a\}$. It is known that $a$ is von Neumann regular if, and only if, $a$ is \emph{strongly von Neumann regular} (i.e. $a\in Q(a)^2 (E)$) if, and only if, there exists (a unique) $b\in E$ such
that $Q(a) (b) =a,$ $Q(b) (a) =b$ and $[Q(a),Q(b)]:=Q(a)\,Q(b) - Q(b)\, Q(a)=0$ if, and only if, $Q(a)$ is norm-closed (compare \cite[Theorem 1]{FerGarSanSi}, \cite[Lemma 3.2, Corollary 3.4,
Proposition 3.5, Lemma 4.1]{Ka96}, \cite[Theorem 2.3, Corollary 2.4]{BurKaMoPeRa08}). The unique element $b$ given above is denoted by $a^{\wedge}$. The set of all von Neumann regular elements in $E$ is denoted by $E^{\wedge}$.\smallskip

Let us recall that an element $a$ in a unital Jordan Banach algebra $J$ is called invertible whenever there exists $b\in J$ satisfying $a \circ b = 1$ and $a^2 \circ b = a.$ Under the above circumstances, the element $b$ is unique and will be denoted by $a^{-1}$. The symbol $J^{-1}= \hbox{inv}(J)$ will denote the set of all invertible elements in $J$. It is well known in Jordan theory that $a$ is invertible if, and only if, the mapping $x\mapsto U_a (x):= 2 (a\circ x)\circ a - a^2 \circ x$ is invertible in $L(J)$, and in that case $U_a^{-1}= U_{a^{-1}}$ (see, for example \cite[\S 4.1.1]{CaRod14}).\smallskip

The notion of invertibility in the Jordan setting provides an adequate point of view to study regularity. More concretely, it is shown in \cite{Ka96}, \cite[Lemma 3.2]{Ka2001} and \cite[Proposition 2.2 and proof of Theorem 3.4]{BurKaMoPeRa08} that an element $a$ in a JB$^*$-triple $E$ is von Neumann regular if and only if there exists a tripotent $v\in E$ such that $a$ is a positive and invertible element in the JB$^*$-algebra $E_2(e)$. It is further known that $a^{\wedge}$ is precisely the (Jordan) inverse of $a$ in $E_2 (v)$.

\section{Pointwise-generalized-inverses}

Our first lemma gathers some basic properties of pg-inverses.

\begin{lemma}\label{l basic pg inv}  Let $\Phi:A\rightarrow B$ be a linear map between complex Banach algebras admitting a
pg-inverse $\Psi$. Then the following statements hold:
\begin{enumerate}[$(a)$]
\item $\Phi$ maps regular elements in $A$ to regular elements in $B$, that is, $\Phi$ is a weak regular preserver, More concretely, if $b$ is a generalized inverse of $a$ then $\Psi(b)$ is a generalized inverse of $\Phi(a)$;
   \item If $A$ is unital and $(\Phi,\Psi)$ is Jordan-triple multiplicative, then $\ker(\Phi)=\ker(\Psi);$
  \item If $A$ is unital and $(\Phi,\Psi)$ is Jordan-triple multiplicative, then $\Phi$ is norm continuous if and only if $\Psi$ is norm continuous;
  \item If $A$ and $B$ are unital and $\Phi(1)\in B^{-1},$ then $\Psi=R_{\Phi(1)^{-1}}\circ L_{\Phi(1)^{-1}}\circ \Phi$ is the unique pg-inverse of $\Phi;$
  \item Let $\Phi_1 : C\to A$ and $\Phi_2: B\to C$ be linear maps admitting a pg-inverse, where $C$ is a Banach algebra, then $\Phi_2 \Phi$ and $\Phi\Phi_1$ admit a pg-inverse too. In particular, if $A$ and $B$ are C$^*$-algebras, then the maps $x\mapsto\Phi(x)^*$, $x\mapsto\Phi(x^*)$, and $x\mapsto\phi(x^*)^*$ admit pg-inverses.
\end{enumerate}
\end{lemma}

\begin{proof}$(a)$ Suppose that $a$ is a regular element in $A$ and let $b$ be a generalized inverse of $a$. Then $\Phi (a) = \Phi (a ba) = \Phi(a) \Psi (b)\Phi(a)$, and hence $\Phi(a)$ is regular too.\smallskip

$(b)$ The conclusion follows from the identities $\Phi (x) = \Phi(1) \Psi(x) \Phi(1)$ and  $\Psi (x) = \Psi(1) \Phi(x) \Psi(1)$ ($x\in A$). Statement $(c)$ can be proved from the same identities.\smallskip

$(d)$ is left to the reader.\smallskip

$(e)$ Suppose that $\Phi_1 : C\to A$ admits a pg-inverse $\Psi_1$. Then $$\Phi \Phi_1 (a b a ) = \Phi (\Phi_1(a)\Psi_1(b)\Phi_1(a))= \Phi(\Phi_1)(a) \Psi(\Psi_1(b)) \Phi(\Phi_1)(a),$$ for all $a,b\in C$. The rest is clear.\end{proof}

In the hypothesis of the above lemma, let us observe that a pg-inverse of a continuous linear operator $\Phi : A\to B$ need not be, in general, continuous. Take, for example, two infinite dimensional Banach algebras $A$ and $B$, a continuous homomorphism $\pi : A \to B$ and an unbounded linear mapping $F: A\to B$. We define $\Phi, \Psi : A\oplus^{\infty} A\to B\oplus^{\infty} B$, $\Phi (a_1,a_2) = (\pi (a_1),0)$ and $\Psi (\pi(a_1), F(a_2))$. Clearly, $\Psi$ is unbounded and $\Phi (a_1,a_2) \Psi (b_1,b_2) \Phi (a_1,a_2) = \Phi ((a_1,a_2) (b_1,b_2) (a_1,a_2)$.\smallskip

We have just seen that every linear map admitting a pg-inverse is a weak regular preserver. The Example \ref{example BurMarMoPe} below shows that reciprocal implication is not always true.\smallskip

The following technical lemma isolates an useful property of linear maps admitting a pg-inverse.

\begin{lemma}\label{lem1} Let $\Phi:A\rightarrow B$ be a liner map between complex Banach algebras, where $A$ is unital. Suppose that $\Psi:A\to B$ is a pg-inverse of $\Phi$, and $z$ is a generalized inverse of $\Phi(1).$ Then we have $$\Phi=L_{(\Phi(1)z)}\circ
\Phi=R_{(z\Phi(1))}\circ \Phi.$$
\end{lemma}

\begin{proof} Since $z$ is a generalized inverse of $\Phi(1)$, the elements $\Phi(1)z$ and $z\Phi(1)$ are idempotents and $\Phi(1)z\Phi(1) = \Phi (1)$. For each $x\in A$ we have $$2 \Phi (x)= \Phi (11x + x 11)  = \Phi(1) \Psi (1) \Phi(x) + \Phi(x) \Psi(1) \Phi (1)$$ $$= \Phi(1) z \Phi(x) + \Phi(x) z \Phi (1).$$ This implies that $\Phi (x) = (\Phi(1) z) \Phi(x) = \Phi(x) (z \Phi (1)).$
\end{proof}

It is not obvious that a linear map admitting a pg-inverse also admits a normalized-pg-inverse. We can conclude now that if the domain is a unital Banach algebra then the desired statement is always true.

\begin{proposition}\label{prop1} Suppose that $A$ is a unital Banach algebra. Let $\Phi:A\rightarrow B$ a linear map admitting a
pointwise-generalized-inverse. Then $\Phi$ has a normalized-pg-inverse. More concretely, if $z$ is a generalized inverse of $\Phi(1)$, the mapping $\Theta = L_{z}\circ R_{z}\circ\Phi$ is a normalized-pg-inverse of $\Phi$.
\end{proposition}

\begin{proof} Let $z$ be a generalized inverse of $\Phi(1),$ and let $\Psi$ be a pg-inverse of $\Phi.$
We set $\Theta =L_{z}\circ R_{z}\circ\Phi.$ By applying Lemma \ref{lem1} with $x=z$, we get
$$ \Theta(aba)  =z \Phi(aba)z =z \Phi(a)\Psi(b)\Phi(a)z $$ $$= z \Big(\Phi(a)z\Phi(1)\Big)\Psi(b) \Big(\Phi(1)z\Phi(a)\Big)z$$
$$= \Big( z \Phi(a)z\Big)\Big(\Phi(1)\Psi(b) \Phi(1)\Big) \Big(z\Phi(a)z\Big)= \Theta (a)\Phi(b)\Theta(a).$$
On the other hand, we also have
$$ \Phi(aba)  = \Phi(a)\Psi(b)\Phi(a)=  \Big(\Phi(a)z\Phi(1)\Big)\Psi(b)\Big(\Phi(1)z\Phi(a)\Big)$$
$$= \Phi(a)\Big(z(\Phi(1)\Psi(b)\Phi(1))z\Big)\Phi(a)= \Phi(a)\Big(z\Phi(b)z\Big)\Phi(a)$$ $$= \Phi(a)\Theta(b)\Phi(a).$$
\end{proof}

Let $A$ and $B$ be complex Banach algebras. We recall that a linear map $T : A\to B$ is called a Jordan homomorphism if $T(a^2) = T(a)^2$ for every $a\in A$, or equivalently, $T(a \circ b ) = T(a)\circ T(b)$, where $\circ$ denotes the natural Jordan product defined by $x\circ y := \frac12 ( xy + yx)$. For each $a$ in $A$ the mapping $U_a : A\to A$ is given by $U_a (x) := 2(a\circ x) \circ a - a^2 \circ x = a x a$. It is well known that a Jordan homomorphism satisfies the identity $T(aba)= T(U_a (b))= U_{T(a)} (T(b)) = T(a)T(b)T(a),$ for all $a, b\in A.$\smallskip

We can now complete the statement in the above proposition. If $\Psi:A\rightarrow B$ is normalized-pg-inverse of a linear mapping $\Phi:A\rightarrow B$, by Proposition \ref{prop1} $\Psi (1)$ is a generalized inverse of $\Phi(1)$ and we clearly have $$\Psi (x) = \Psi (1) \Phi (x) \Psi (1),$$ for all $x\in A$.

\begin{lemma}\label{lem24}
Let $\Phi,\Psi:A\rightarrow B$ be linear maps between Banach algebras, with $A$ unital. Suppose that $(\Phi,\Psi)$ is Jordan-triple multiplicative. Then the following statements hold:
\begin{enumerate}[$(a)$]
\item The identities
$${\Psi(1)}  \Phi (a)= \Psi (a) \Phi(1), \quad
\Phi (a) \Psi(1)=\Phi(1) \Psi (a),$$
$$\Phi(a)\Psi(b)=\Phi(1)\Psi(a)\Phi(b)\Psi(1), \hbox{ and } \Psi(1) \Phi(a)\Psi(b) \Phi(1) =\Psi(a)\Phi(b),$$ hold for all $a,b\in A;$
 \item The linear maps $T=L_{\Psi(1)}\circ \Phi$ and $S=R_{\Psi(1)}\circ \Phi$ are Jordan homomorphisms satisfying:
    $$\Phi(a)\Psi(b)=S(a)S(b),\text{ and }\Psi(a)\Phi(b)=T(a)T(b),$$ for all $a,b\in A.$
\end{enumerate}
\end{lemma}

\begin{proof}$(a)$ We know from previous results that $\Phi (a) = \Phi (1) \Psi (a) \Phi (1),$ and $\Psi (a) = \Psi (1) \Phi (a) \Psi (1),$ for all $a\in A$, where $\Phi(1)$ is a normalized generalized inverse of $\Psi(1)$. We conclude from Lemma \ref{lem1} that $${\Psi(1)}  \Phi (a) = \Psi (1) \Phi (1) \Psi (a) \Phi (1) = \Psi (a) \Phi (1),$$ and $$ \Phi (a) \Psi(1) = \Phi (1) \Psi (a) \Phi (1) \Psi(1) = \Phi (1) \Psi (a),$$ for all $a\in A$. Consequently, $$ \Phi(a)\Psi(b) = \Phi (1) \Psi (a) \Phi (1) \Psi (1) \Phi (b) \Psi (1) = \Phi (1) \Psi (a)  \Phi (b) \Psi (1).$$ The remaining identity follows by similar arguments.\smallskip

$(b)$ With the notation above, $T(a) T(b) = \Psi(1) \Phi(a) \Psi(1) \Phi(b) = \Psi(a) \Phi(b),$ and consequently, $$2 T(a^2)= 2 \Psi(1) \Phi(a^2) = \Psi(1) \Phi(aa 1 + 1aa)$$ $$ = \Psi(1) \Phi(a)\Psi(a)\Phi(1) + \Psi(1) \Phi(1)\Psi(a)\Phi(a) = 2 \Psi(a)\Phi(a) = 2 T(a)^2.$$ The rest is left to the reader.
\end{proof}

The previous properties are now subsumed in an equivalence.

\begin{proposition}\label{prop2}
Let $\Phi:A\rightarrow B$ be a linear map between complex Banach algebras with $A$ unital. Then the following statements are equivalent:
\begin{enumerate}[$(a)$]
    \item $\Phi$ admits a normalized-pg-inverse;
    \item There exists a Jordan homomorphism  $T:A\rightarrow B$ such that $\Phi= R_{\Phi(1)}\circ T$ and $\Phi(1) B=T(1) B;$
    \item There exists a Jordan homomorphism $S:A\rightarrow B$ such that $\Phi=L_{\Phi(1)}\circ S$ and $B \Phi(1)=B S(1).$
\end{enumerate}
\end{proposition}

\begin{proof} $(a)\Rightarrow (b)$ Suppose that $\Phi$ admits a normalized-pg-inverse $\Psi:A\rightarrow B$. By Lemma \ref{lem24} the mapping $T=L_{\Phi(1)}\circ \Psi$ is a Jordan homomorphism and $R_{\Phi(1)}\circ T (a) = \Phi(1) \Psi(a) \Phi(1)  = \Phi (a),$ or every $a\in A$. On the other hand, $T(1)=\Phi(1)\Psi(1)$ is an idempotent in $B$ and $T(1)\Phi(1)=\Phi(1),$ which implies that $T(1) B=\Phi(1)B.$\smallskip

\noindent$(b)\Rightarrow (a)$ Let $T:A\rightarrow B$ be a Jordan homomorphism such that $\Phi = R_{\Phi(1)}\circ T$ and $\Phi(1) B=T(1) B.$ Under these hypothesis, there exists $c\in B$ such that $T(1) = T(1)^2 =\Phi(1) c.$ The element $T(1)$ is an idempotent in $B$ with $T(a)\circ T(1) = T(a)$, for every $a\in A$. Thus, $T(a) = T(1) T(a) = T(a) T(1) = T(1) T(a) T(1)$, for every $a$ in $A$. If we set $\Psi=L_{c}\circ T,$ by applying Lemma \ref{lem1}, we obtain
$$\Phi(aba)= T(aba) \Phi(1)= T(a)T(b)T(a) \Phi(1) = T(a)T(1) T(b)T(a) \Phi(1) $$ $$= T(a)[\Phi(1)c]T(b)T(a)\Phi(1)=\Phi(a)\Psi(b)\Phi(a);\quad \forall\;a,b\in A.$$
The implications $(a)\Rightarrow (c)$ and $(c)\Rightarrow (a)$ follow by similar arguments.
 \end{proof}

\begin{example}\label{example BurMarMoPe}\cite[Remark 5.10]{BurMarMoPe2016} Let $H$ be an infinite dimensional complex Hilbert space, let $v,w$ be (maximal) partial isometries such that $v^*v=1= w^*w$ and $vv^*\perp ww^*$. We set $A= \mathbb{C}\oplus^{\infty}\mathbb{C},$ and consider the operator $T: A \to B(H)$ given by $$T(\lambda,\mu) = \frac{\lambda}{2} (v+w) + \frac{\mu}{2} (v-w).$$ It is shown in \cite[Remark 5.10]{BurMarMoPe2016} that $T$ maps extreme point of the closed unit ball of $A$ to extreme point of the closed unit ball of $B(H),$ but $T$ does not preserves Moore-Penrose inverses strongly, that is, $T(a^{\dag})\neq T(a)^{\dag}$ for every Moore-Penrose invertible element $a\in A$.\smallskip

Let us show that $T$ is a weak preserver, that is, $T$ maps regular elements to regular elements. It is easy to check that an element $a=(\lambda,\mu)\in A$ is regular if and only if it is Moore-Penrose invertible if and only if $|\lambda|+|\mu|\neq 0$ (i.e. $a\neq 0$), and in such a case $a^{\dag} = (\lambda^{-1},0)$ if $\mu=0$, $a^{\dag} = (0,\mu^{-1})$ if $\lambda=0$ and $a^{\dag}=a^{-1}$ otherwise. 
Given $\lambda, \mu\in \mathbb{C}$ we have $$T(a)^* T(a) = \left( \frac{\overline{\lambda}}{2} (v+w)^* + \frac{\overline{\mu}}{2} (v-w)^*\right) \left(\frac{\lambda}{2} (v+w) + \frac{\mu}{2} (v-w)\right) $$ $$= \left(\frac{|\lambda|^2}{4} + \frac{|\mu|^2}{4}\right) (v^*v+w^* w)= \left(\frac{|\lambda|^2}{4} + \frac{|\mu|^2}{4}\right) 1,$$ which assures that $T(a)$ admits a Moore-Penrose inverse.\smallskip

We shall finally show that $T$ does not admit a pg-inverse. Arguing by contradiction, we assume that $T$ admits a pg-inverse. Proposition \ref{prop1} assures that $T$ admits a normalized-pg-inverse and Proposition \ref{prop2}$(c)$ implies the existence of a Jordan homomorphism $J: A\to B(H)$ such that $T(a) = T(1) J (a),$ for every $a\in A$. Having in mind that $T(1) = T(1,1) = v$, we have $T(\lambda, \mu) = v J(\lambda,\mu),$ for every $\lambda,\mu\in \mathbb{C}.$ Therefore $J(\lambda,\mu) = v^* v J(\lambda,\mu) = v^* T(\lambda,\mu)$, for every $\lambda,\mu\in \mathbb{C}$, and thus $$ \frac{\lambda^2 +\mu^2}{2} 1 =  v^* (\frac{\lambda^2 +\mu^2}{2} v + \frac{\lambda^2 -\mu^2}{2} w)=v^* T(\lambda^2,\mu^2) =v^* T((\lambda,\mu)^2) $$ $$= (v^* T(\lambda,\mu)) (v^* T(\lambda,\mu)) = v^* (\frac{\lambda +\mu}{2} v + \frac{\lambda -\mu}{2} w) v^* (\frac{\lambda +\mu}{2} v + \frac{\lambda -\mu}{2} w)$$ $$= \frac{\lambda +\mu}{2} 1 \frac{\lambda +\mu}{2} 1= \frac{(\lambda +\mu)^2}{4} 1,$$ for every $\lambda,\mu\in \mathbb{C}$, which is impossible.
\end{example}

It is known that we can find an infinite dimensional complex Banach algebra $A$ and an unbounded homomorphism $\pi : A\to \mathbb{C}$. Clearly $\pi$ admits a normalized-pg-inverse but it is not continuous. However, every homomorphism $\pi$ from an arbitrary
complex Banach algebra $A$ into a C$^*$-algebra $B$ whose image is a $^*$-subalgebra of $B$ is automatically continuous (see \cite[Theorem 4.1.20]{Rick60}).

We can relax the hypothesis of $A$ being unital at the cost of assuming the continuity of $\Phi$ and $\Psi$. Henceforth, the bidual of a Banach space $X$ will be denoted by $X^{**}$.

\begin{lemma}\label{l bitransposed} Let $\Phi,\Psi:A\rightarrow B$ be continuous linear maps between C$^*$-algebras. Suppose that $\Psi$ is a {\rm(}normalized-{\rm)}pg-inverse of $\Phi$. Then $\Psi^{**}: A^{**}\to B^{**}$ is a {\rm(}normalized-{\rm)}pg-inverse of $\Phi^{**}$.
\end{lemma}

\begin{proof} The maps $\Phi^{**},\Psi^{**}: A^{**}\to B^{**}$ are weak$^*$-to-weak$^*$ continuous operators between von Neumann algebras. We recall that, by Sakai's theorem (see \cite[Theorem 1.7.8]{S}), the products of $A^{**}$ and $B^{**}$ are separately weak$^*$-continuous. Let us fix $a,b,c\in A^{**}$. By Goldstine's theorem we can find three bounded nets $(a_{\lambda})$, $(b_{\mu})$ and $(c_{\delta})$ in $A$ converging in the weak$^*$-topology of $A^{**}$ to $a,b$ and $c$, respectively. By hypothesis, $$\Phi (a_{\lambda} b_{\mu}  c_{\delta}+  c_{\delta} a_{\lambda} b_{\mu}) = \Phi (a_{\lambda}) \Psi( b_{\mu}) \Phi(c_{\delta})+ \Phi(c_{\delta}) \Phi (a_{\lambda}) \Psi( b_{\mu}),$$ 
for every $\lambda,\mu$ and $\delta$. Taking weak$^*$-limits in $\lambda$, $\mu$ and $\delta$ we get $$\Phi^{**} (abc + cba) = \Phi^{**} (a) \Psi^{**}( b) \Phi^{**} (c)+ \Phi^{**}(c) \Phi^{**} (a) \Psi^{**}( b).$$ 
\end{proof}

Combining Proposition \ref{prop2} with Lemma \ref{l bitransposed} we get the following.

\begin{corollary}\label{c prop2 continuity instead unital}
Let $\Phi:A\rightarrow B$ be a continuous linear operator between C$^*$-algebras. Then the following statements are equivalent:
\begin{enumerate}[$(a)$]
    \item $\Phi$ admits a continuous normalized-pg-inverse;
    \item There exists a continuous Jordan homomorphism  $T:A^{**}\rightarrow B^{**}$ such that $\Phi= R_{\Phi^{**}(1)}\circ T$ and $\Phi^{**}(1) B^{**}=T(1) B^{**};$
    \item There exists a continuous Jordan homomorphism $S:A^{**}\rightarrow B^{**}$ such that $\Phi=L_{\Phi^{**}(1)}\circ S$ and $B^{**} \Phi^{**}(1)=B^{**} S(1).$ $\hfill\Box$
\end{enumerate}
\end{corollary}

Let $A$ and $B$ be C$^*$-algebras. We recall that a linear mapping $T :A\to B$ strongly preserves Moore-Penrose invertibility (respectively, invertibility) if for each Moore-Penrose invertible (respectively, invertible) element $a\in A,$ the element $T(a)$ is Moore-Penrose invertible (respectively, invertible) and we have $T(a^{\dag}) = T(a)^{\dag}$ (respectively, $T(a^{-1})=T(a)^{-1}$). Hua's theorem (see \cite{Hua}) affirms that every unital additive map between skew fields that strongly preserves invertibility is either an isomorphism or an anti-isomorphism. Suppose $A$ is unital. In this case M. Burgos, A.C. M\'{a}rquez-Garc\'{i}a and A. Morales-Campoy establish in \cite[Theorem 3.5]{BurMarMor12} that a linear map $T: A\to B$ strongly preserves Moore-Penrose invertibility if, and only if, $T$ is a Jordan $^*$-homomorphism $S$ multiplied by a partial isometry $e$ in $B$ such that $T(a) = ee^* T(a) e^*e$ for all $a\in A$, if and only if, $T$ is a triple homomorphism (i.e. $T$ preserves triple products of the form $\{a,b,c\}:=\frac12 (ab^* c+ cb^*a)$). The problem for linear maps strongly preserving Moore-Penrose invertibility between general C$^*$-algebras remains open.\smallskip

Let $T: A\to B$ be a triple homomorphism between C$^*$-algebras. In this case $$T(a b a) = T(\{a,b^*, a\}) =\{T(a), T(b^*), T(a)\} = T(a) T(b^*)^* T(a),$$ and $$T(a^*)^* T(b) T(a^*)^* =\{T(a^*)^*, T(b)^*, T(a^*)^*\} =  \{T(a^*),T(b),T(a^*)\}^* $$ $$= T(\{a^*,b,a^*\})^*= T(a^* b^* a^*)^* = T((aba)^*)^*,$$  for all $a,b\in A.$ These identities show that $x\mapsto T(x^*)^*$ is a normalized-pg-inverse of $T$. So, when $A$ is unital, it follows from the results by Burgos, M\'{a}rquez-Garc\'{i}a and Morales-Campoy that every linear map $T : A\to B$ strongly preserving Moore-Penrose invertibility admits a normalized-pg-inverse. However, the class of linear maps admitting a normalized-pg-inverse is strictly bigger than the class of linear maps strongly preserving Moore-Penrose invertibility. For example, let $z$ be an invertible element in $B(H)$ with $z^*\neq z$, the mapping $T: B(H)\to B(H),$ $T(x) = z x z^{-1}$ is a homomorphism, and hence a Jordan homomorphism and does not strongly preserve Moore-Penrose invertibility.\smallskip

We recall that an element $e$ in a C$^*$-algebra $A$ is a \emph{partial isometry} if $e e^* e =e$. Let us observe that a C$^*$-algebra might not contain a single partial isometry. However, a famous result due to Kadison shows that the extreme points of the closed unit ball of a C$^*$-algebra $A$ are precisely the maximal partial isometries in $A$ (see \cite[Proposition 1.6.1 and Theorem 1.6.4]{S}). Therefore, every von Neumann algebra contains an abundant set of partial isometries. When a C$^*$-algebra $A$ is a regarded as a JB$^*$-triple with respect to the product given by $\{a,b,c\} = \frac12 (a b^* c + c b^* a)$, partial isometries in $A$ are exactly the fixed points of the this triple product and are called \emph{tripotents}.\smallskip

Suppose that $e$ and $v$ are non-zero partial isometries in a C$^*$-algebra $A$ such that $ e v e = e$ and $v = v e v$. Then $e =  (e e^*) v^* (e^* e)$ and $v = (vv^*) e^* (v^* v)$. This implies, in the terminology of \cite{FriRu85}, that $P_2 (e) (v^*)= (ee^*) v^* (e^*e) = e$. Since $v$ is a norm-one element, we can conclude from \cite[Lemma 1.6 or Corollary 1.7]{FriRu85} that $v^* = e + (1-ee^*) v^* (1-e^*e).$ However the identity $v = v e v $ implies that $v= e^*$.

\begin{theorem}\label{thm2} Let $\Phi, \Psi:A\rightarrow B$ be linear maps between C$^*$-algebras. Suppose that $(\Phi,\Psi)$ is Jordan-triple multiplicative. Then the following are equivalent:
\begin{enumerate}[$(a)$]
\item $\Phi$ and $\Psi$ are contractive;
\item  $\Psi(a)=\Phi(a^*)^*,$ for every $a\in A;$
\item $\Phi$ and $\Psi$ are triple homomorphisms.
\end{enumerate}
\end{theorem}

\begin{proof} $(a)\Rightarrow (b)$ Clearly $\Phi^{**}$ and $\Psi^{**}$ are contractive operators and by Lemma \ref{l bitransposed}, $\Psi^{**}$ is a normalized-pg-inverse of $\Phi^{**}$. Let $e$ be a partial isometry in $A^{**}$. Since
\begin{equation}\label{eq bitransposed at e} \Phi^{**}(e)=\Phi^{**}(e)\Psi^{**}(e^*)\Phi^{**}(e),\text{ and }\Psi^{**}(e^*)=\Psi^{**}(e^*)\Phi^{**}(e)\Psi^{**}(e^*),
\end{equation} we deduce that $\Psi^{**}(e^*)$ is a generalized inverse of $\Phi^{**}(e).$ Applying that  $\Phi^{**}$ and $\Psi^{**}$ are contractions, it follows that $\Phi^{**} (e)$ and $\Psi^{**} (e)$ lie in the closed unit ball of $B^{**}$ and admit normalized generalized inverses in the closed unit ball of $B^{**}.$ Corollary 3.6 in \cite{BurKaMoPeRa08} implies that $\Phi^{**} (e)$ and $\Psi^{**}(e)$ are partial isometries in $B^{**}$.  We can now deduce from \eqref{eq bitransposed at e} and the comments preceding this theorem that $\Psi^{**}(e^*)=\Phi^{**}(e)^*.$ In particular, $\Psi(p)=\Phi(p)^*,$ for every projection $p\in A^{**}.$ Since in a von Neumann algebra every self-adjoint element can be approximated in norm by a finite linear combination of mutually orthogonal projections, we get $\Psi^{**}(a)=\Phi^{**}(a)^*,$ for every $a\in A^{**}_{sa},$ and by linearity we have $\Phi^{**} (a)^*=\Psi^{**}(a^*),$ for every $a\in A^{**}$.\smallskip

$(b)\Rightarrow (c)$ Let us assume that $\Psi(a^*)=\Phi(a)^*,$ for every $a\in A.$ In this case $$\Phi \{abc\} = \frac12 \Phi (a b^* c + c b^* a) = \frac12  (\Phi(a) \Psi(b^*) \Phi(c) + \Phi(c) \Psi(b^*) \Phi(a) ) $$ $$= \frac12  (\Phi(a) \Phi(b)^* \Phi(c) + \Phi(c) \Phi(b)^* \Phi(a) ) =\{\Phi(a), \Phi(b), \Phi (c)\},$$ which shows that $\Phi$ (and hence $\Psi$) is a triple homomorphism.\smallskip

The implication $(c)\Rightarrow (a)$ follows form the fact that triple homomorphisms are contractive (see, for example, \cite[Proposition 3.4]{Harris81} or \cite[Lemma 1$(a)$]{BarDanHor88}).
\end{proof}

The fact that every contractive representation of a C$^*$-algebra (equivalently, every contractive homomorphism between C$^*$-algebras) is a $^*$-homo-morphism seems to be part of the folklore in C$^*$-algebra theory (see, for example, the last lines in the proof of \cite[Theorem 1.7]{BlePaul2001}). Actually, every contractive Jordan homomorphism between C$^*$-algebras is a Jordan $^*$-homomorphism\hyphenation{homo-morphism}; However, we do not know an explicit reference for this fact. We present next an explicit argument derived from our results. A generalization for Jordan homomorphisms between JB$^*$-algebras will be established in Corollary \ref{c contractive are symetric in JB*-algebras}.

\begin{corollary}\label{c Jordan homomorphisms} Let $A$ and $B$ be C$^*$-algebras and let $\Phi:A\rightarrow B$ be a Jordan homomorphism.  Then the following statements are equivalent:\begin{enumerate}[$(a)$] \item $\Phi $ is a contraction;
\item $\Phi$ is a symmetric map {\rm(}i.e. $\Phi$ is a Jordan $^*$-homomorphism{\rm)};
\item $\Phi$ is a triple homomorphism.
\end{enumerate}
If $A$ is unital, then the above statements are also equivalent to the following:
\begin{enumerate}[$(c)$]\item $\Phi $ \emph{strongly preserves regularity}.
\end{enumerate}
\end{corollary}

\begin{proof} The implication $(a)\Rightarrow (b)$ is given by Theorem \ref{thm2}. It is known that every Jordan $^*$-homomorphism is a triple homomorphism, then $(b)$ implies $(c)$. Every triple homomorphism is continuous and contractive (see \cite[Lemma 1$(a)$]{BarDanHor88}), and hence $(c)\Rightarrow (a)$.\smallskip

The final statement follows from \cite[Theorem 3.5]{BurMarMor12}.\end{proof}

It seems appropriate to clarify the connections between Corollary \ref{c Jordan homomorphisms} and previous results.  It is known that every triple homomorphism between general C$^*$-algebras strongly preserves regularity (compare \cite{BurMarMor12} and \cite{BurMarMoPe2016}). Actually, if $A$ and $B$ are C$^*$-algebras with $A$ unital, and $T : A \to B$ is a linear map, then by \cite[Theorem 3.5]{BurMarMor12}, $T$ strongly preserves regularity if, and only if, $T$ is a triple homomorphism. So, in case $A$ being unital the equivalence $(c)\Leftrightarrow (d)$ in Corollary \ref{c Jordan homomorphisms} can be established under weaker hypothesis. For a non-unital C$^*$-algebra $A$ the continuity of a linear mapping $T: A\to B$ strongly preserving regularity does not follow automatically. For example, by \cite[Remark 4.2]{BurMarMoPe2016}, we know the existence of an unbounded linear mapping $T: c_0 \to c_0$ which strongly preserves regularity. According to our knowledge, it is an open problem whether every continuous linear map strongly preserving regularity between general C$^*$-algebras is a triple homomorphism.\smallskip


\section{Orthogonality preservers and non-unital versions}\label{sec 3}

Let $A$ be a C$^*$-algebra. We recall that an \emph{approximate unit} of $A$ is a net $(u_{\lambda})$ such that $0\leq u_{\lambda}\leq 1$ for every $\lambda$, $u_{\lambda}\leq u_{\mu}$ for every $\lambda\leq \mu,$ and $$\lim_{\lambda} \| x - x u_{\lambda}\| = \lim_{\lambda} \| x - u_{\lambda} x \| = \lim_{\lambda} \| x -  u_{\lambda} x u_{\lambda}\| =0, $$ for every $x\in A$. Every C$^*$-algebra admits an approximate unit (see \cite[Theorem 3.1.1]{Murph}).\smallskip

Let $(u_{\lambda})$ be an approximate unit in a C$^*$-algebra $A$, and let us regard $A$ as a C$^*$-subalgebra of $A^{**}$. Having in mind that a functional $\phi$ in $A^{*}$ is positive if and only if $\|\phi\| = \lim_{\lambda} \phi (u_{\lambda})$ (see \cite[Theorem 3.3.3]{Murph}), we can easily see that $(u_{\lambda})\to 1$ in the weak$^*$ topology of $A^{**}$.

\begin{lemma}\label{l non-unital with projections automatic continuity} Let $\Phi, \Psi:A\rightarrow B$ be linear maps between C$^*$-algebras. Suppose that $\Phi$ is continuous and $(\Phi,\Psi)$ is Jordan-triple multiplicative. Then the following statements hold:
\begin{enumerate}[$(a)$]
\item $\Phi^{**} (a b c + c b a ) = \Phi^{**} (a) \Psi(b) \Phi^{**} (c) + \Phi^{**} (c) \Psi(b) \Phi^{**} (a)$ for every $a,c$ in $A^{**}$, and every $b$ in $A$;
\item $\Phi (b) = \Phi^{**} (1) \Psi(b) \Phi^{**} (1) $ for every $b$ in $A$;
\item The mapping $T: A\to B^{**}$, $T(x) = \Phi^{**} (1) \Psi (x)$ satisfies $T(a) T(b) = \Phi (a) \Psi(b)$, and $\Phi (a) = T(a) \Phi^{**} (1)$, for every $a,b\in A$;
\item The mapping $S: A\to B^{**}$, $S(x) =\Psi (x) \Phi^{**} (1) $ satisfies $S(a) S(b) = \Psi (a) \Phi(b)$, and $\Phi (a) =  \Phi^{**} (1)  S(a) $, for every $a,b\in A$;
\item  Suppose that $p$ and $q$ are projections in $A$ with $p q =0$, then $T(p) T(q)  =S(p) S(q)  =0$, where $T$ and $S$ are the maps defined in previous items.
\end{enumerate}
\end{lemma}

\begin{proof}$(a)$ Applying that $\Phi$ is continuous, the bitransposed map $\Phi^{**} : A^{**}\to B^{**}$ is weak$^*$-continuous. Let $a$ and $c$ be elements in $A^{**}$, and let $b\in A$. By Goldstine's theorem we can find bounded nets $(a_\lambda)$ and $(c_\mu)$ in $A$ converging, in the weak$^*$ topology of $A^{**}$, to $a$ and $c$, respectively. By hypothesis $$\Phi (a_\lambda  b c_\mu +  c_\mu b a_\lambda)  = \Phi (a_\lambda) \Psi(b) \Phi(c_\mu) + \Phi (c_\mu) \Psi(b) \Phi(a_\lambda),$$ for every $\lambda,\mu$. Since the product of $A^{**}$ is separately weak$^*$ continuous, the weak$^*$-continuity of $\Phi^{**}$ implies that $$\Phi^{**} (a b c + c b a ) = \Phi^{**} (a) \Psi(b) \Phi^{**} (c) + \Phi^{**} (c) \Psi(b) \Phi^{**} (a).$$

$(b)$ Follows from $(a)$ with $a=c=1$.\smallskip

$(c)$ By definition and $(b)$ we have $$ T(a) T(b) = \Phi^{**} (1) \Psi (a) \Phi^{**} (1) \Psi (b) = \Phi (a) \Psi (b),$$ and $T(a) \Phi^{**} (1) = \Phi^{**} (1) \Psi(a) \Phi^{**} (1) = \Phi (a)$, for every $a,b\in A$. The proof of $(d)$ is very similar.\smallskip

$(e)$ Let us take two projections $p,q\in A$ with $pq =0$. By definition and $(b)$ or $(c)$ we have  $$T(p) T(q)  = \Phi^{**} (1) \Psi (p) \Phi^{**} (1) \Psi (q) = \Phi(p) \Psi (q)  $$ $$= \Phi(p) \Psi (q) \Phi (q)\Psi (q)= (\Phi (pqq + qqp) - \Phi(q) \Psi (q) \Phi (p) ) \Psi (q) $$ $$ = - \Phi(q) \Psi (q) \Phi (p)  \Psi (q) = - \Phi(q) \Psi (qp q) =0.$$
\end{proof}

Let us explore some of the questions posed before. In our first proposition we shall prove that the normalized-pg-inverse of a continuous linear map on $c_0$ is automatically continuous.

\begin{proposition}\label{p automatic continuity on c0} Let $\Phi, \Psi: c_0 \to c_0$ be linear maps such that $\Phi$ is continuous and $(\Phi,\Psi)$ is Jordan-triple multiplicative. Then $\Psi$ is continuous.
\end{proposition}

\begin{proof} We can assume that $\Phi,\Psi\neq 0$. Let $(e_n)$ be the canonical basis of $c_0$.  Applying the previous Lemma \ref{l non-unital with projections automatic continuity}$(c)$, the mapping  $T: c_{0}\to c_{0}^{**}=\ell_\infty$, $T(x) = \Phi^{**} (1) \Psi (x)$ satisfies $T(a) T(b) = \Phi (a) \Psi(b)$, and $\Phi (a) = T(a) \Phi^{**} (1)$, for every $a,b\in c_{0}$. By the just quoted lemma, $T(p) T(q) =0$ for every pair of projections $p,q\in c_0$ with $p q =0$, and consequently,  $$\Phi (p ) \Phi (q) = T(p) \Phi^{**} (1) T(q) \Phi^{**} (1) = T(p) T(q) \Phi^{**} (1) \Phi^{**} (1) =0.$$ We can therefore conclude that $\Phi (e_n) \Phi (e_m)= 0$ for every $n\neq m$ in $\mathbb{N}$. Since $\Phi (e_n) = \Phi (e_n) \Psi(e_n) \Phi (e_n)$ and $\Psi (e_n)  = \Psi (e_n) \Phi (e_n) \Psi (e_n)$, we deduce that $\Phi(e_n)$ and $\Psi(e_n)$ both are regular elements in $c_0$ and $\Phi (e_n)$ is a normalized generalized inverse of $\Psi (e_n)$. Therefore, for each natural $n$ with $\Phi (e_n)\neq 0$ there exists a finite subset $\hbox{supp} (\Phi(e_n)) = \{k^n_1,\ldots,k^{n}_{m_n}\}\subset \mathbb{N}$ and non-zero complex numbers $\{\lambda_j^n : j\in \hbox{supp} (\Phi(e_n)) \}$ with the following properties: $|\lambda_j^n|\leq \|\Phi\|$ for every $j\in \hbox{supp} (\Phi(e_n))$ and every natural $n$,   $$\hbox{supp} (\Phi(e_n)) \cap \hbox{supp} (\Phi(e_m))= \emptyset,\hbox{ for all } n\neq m,$$ and $$\Phi (e_n ) = \sum_{j\in \hbox{supp} (\Phi(e_n))} \lambda_j^n e_j, \hbox{ and } \Psi (e_n ) =\sum_{j\in \hbox{supp} (\Phi(e_n))} \frac{1}{\lambda_j^{n}} e_j, \ \ \ \forall n\in \mathbb{N}.$$

Let us observe that $\|\Psi(e_n) \| = \max\{ \frac{1}{|\lambda_j^{n}|} : j \in \hbox{supp} (\Phi(e_n)) \}$. To simplify the notation, let $j (n)\in \hbox{supp} (\Phi(e_n))$ be an element satisfying $\frac{1}{|\lambda_{j(n)}^{n}|} = \|\Psi(e_n) \|$.\smallskip

We claim that the set $\{ \|\Psi(e_n) \|: n\in \mathbb{N}\}$ must be bounded. Otherwise, we can find a subsequence $( \|\Psi(e_{\sigma(n)}) \|)$ satisfying $\frac{1}{|\lambda_{j(\sigma(n))}^{\sigma(n)}|} = \|\Psi(e_{\sigma(n)}) \| > n$ for every natural $n$. Let $\pi_2: c_0 \to c_0$ be the natural projection of $c_0$ onto the C$^*$-subalgebra generated by the elements $\{e_{j(\sigma(n))} : n\in \mathbb{N} \}$, and let $\iota : c_0 = \overline{span} \{ e_{\sigma(n)} : n\in \mathbb{N}\}\to c_0$ denote the natural inclusion. The maps $\Phi_1= \pi_2 \Phi \iota, \Psi_1=\pi_2 \Psi \iota: c_0\to c_0$ are linear maps, $\Psi$ is a normalized-pg-inverse of $\Phi,$ the latter is continuous, $\Psi_1 (e_{\sigma(n)})= \frac{1}{\lambda_{j(\sigma(n))}^{\sigma(n)}} e_{j(\sigma(n))} $, and $\Phi_1 (e_{\sigma(n)})= {\lambda_{j(\sigma(n))}^{\sigma(n)}} e_{j(\sigma(n))}$. The element $\displaystyle a = \sum_{m\in \mathbb{N}} {\lambda_{j(\sigma(m))}^{\sigma(m)}} e_{j(\sigma(m))}$ lies in $c_0$ and $\|\Psi_1 (a)\| <\infty$. Therefore $\displaystyle \Psi_1 (a) = \sum_{m\in \mathbb{N}} {\mu_{m}} e_{j(\sigma(m))}$ for a unique sequence $(\mu_m)\to 0$. Let us write $j ({\sigma} (n)) = j_1(n)$. Under these conditions $${\lambda_{j_1(n)}^{\sigma(n)}} \mu_n e_{j_1(n)} = \Psi_1 (a) \Phi_1 (e_{j_1(n)}) = \Psi_1 (a) \Phi_1 (e_{j_1(n)}) \Psi_1 (e_{j_1(n)}) \Phi_1 (e_{j_1(n)})$$ $$= \Big(\Psi_1 ( e_{j_1(n)} e_{j_1(n)} a + a e_{j_1(n)}e_{j_1(n)})- \Psi_1 (e_{j_1(n)})  \Phi_1 (e_{j_1(n)}) \Psi_1 (a) \Big) \Phi_1 (e_{j_1(n)})$$ $$ = \Psi_1 ( 2 {\lambda_{j_1(n)}^{\sigma(n)}} e_{j_1(n)}) \Phi_1 (e_{j_1(n)}) - \Psi_1 (e_{j_1(n)})  \Phi_1 (e_{j_1(n)}) \Psi_1 (a)\Phi_1 (e_{j_1(n)}) $$ $$= 2 {\lambda_{j_1(n)}^{\sigma(n)}} \Psi_1 (  e_{j_1(n)}) \Phi_1 (e_{j_1(n)}) - \Psi_1 (e_{j_1(n)})  \Phi_1 (e_{j_1(n)} a e_{j_1(n)}) $$ $$ =  2 {\lambda_{j_1(n)}^{\sigma(n)}} e_{j_1(n)}- \Psi_1 (e_{j_1(n)})  \Phi_1 ( {\lambda_{j_1(n)}^{\sigma(n)}} e_{j_1(n)}) =  {\lambda_{j_1(n)}^{\sigma(n)}} e_{j_1(n)},$$ which proves that $\mu_n =1$ for all $n,$ leading to a contradiction.\smallskip

Let $M$ be a positive bound of the set $\{ \|\Psi(e_n) \|: n\in \mathbb{N}\}$. For each natural $n$, we set $q_n :=\sum_{k=1}^{n} e_k$. Clearly, $(q_n)$ is an approximate unit in $c_0$. Since for each $n\neq m$ we have $\Phi (e_n) \Phi(e_m)=0$ (i.e., $\hbox{supp} (\Phi(e_n)) \cap \hbox{supp} (\Phi(e_m))= \emptyset$), and, for each natural $j$, $\Phi (e_j)$ is a normalized generalized inverse of $\Psi (e_j)$, we deduce that $\Psi (e_n) \Psi(e_m)=0$ (i.e., $\hbox{supp} (\Psi(e_n)) \cap \hbox{supp} (\Psi(e_m))= \emptyset$) for every $n\neq m$. Consequently, for each finite subset $F\subseteq \mathbb{N}$ we have \begin{equation}\label{boundedness of Psi qn} \left\|\Psi\left(\sum_{j\in F} e_j\right)\right\|= \max\left\{ \left\|\Psi\left(e_j\right)\right\| : j\in F\right\} \leq M,
\end{equation} and consequently  $\|\Psi(q_n)\|\leq M$, for every natural $n$.\smallskip

We shall prove next that for each $x\in c_0$ we have $$\lim_n (\Psi (x - q_n x))_n= 0.$$ Indeed, let us take $y,z,w\in c_0$ such that $x = y z w$ (in the case of $c_0$ the existence of such $y,z,w$ is almost obvious but we can always allude to Cohen's factorization theorem \cite[Theorem VIII.32.22]{HewRoss}). By assumptions $$ \Psi (x - q_n x) = \Psi (y (1 - q_n) z w ) = \Psi (y) \Phi (z - q_n z) \Psi (w).$$ Since $\Phi$ is continuous and  $((1 - q_n) z)$ tends in norm to $0$, we deduce that $\lim_n (\Psi (x - q_n x))_n= 0$ as we claimed.\smallskip

Finally, for an arbitrary $x$ in the closed unit ball of $c_0$ we have $$\Psi (q_n x) =\Psi (q_n x q_n) = \Psi (q_n) \Phi (x) \Psi (q_n),$$ and hence $\left\| \Psi (q_n x) \right\|\leq M^2 \left\|\Phi\right\|$. The norm convergence of $\Psi (q_n x)$ to $\Psi (x)$, assures that  $\left\| \Psi (x) \right\|\leq M^2 \left\|\Phi\right\|$. The arbitrariness of $x$ proves the continuity of $\Psi$.
\end{proof}

The previous proposition remains valid if $c_0$ is replaced with $c_0 (\Gamma)$.\smallskip

Our next goal is to extend the previous Proposition \ref{p automatic continuity on c0} to linear maps on $K(H)$. For that purpose we isolate first a technical results which is implicit in the proof of the just commented proposition.

\begin{lemma}\label{l characterization of automatic continuity of a npg-inverse with Phi**(1) regular} Let $\Phi, \Psi: A \to B$ be linear maps between C$^*$-algebras such that $\Phi$ is continuous and $(\Phi,\Psi)$ is Jordan-triple multiplicative. Then the following are equivalent:
\begin{enumerate}[$(1)$]
\item $\Phi$ admits a continuous normalized-pg-inverse $\Psi: A\to B^{**}$;
\item $\Phi^{**}(1)$ is a regular element in $B^{**}.$
\end{enumerate}
\end{lemma}

\begin{proof}
$(1)\Rightarrow (2)$ Suppose that  $\Phi$ admits a continuous normalized-pg-inverse $\Psi: A\rightarrow B.$ By Lemma \ref{l bitransposed}, the mapping $\Psi^{**}: A^{**}\to B^{**}$ is a normalized-pg-inverse of $\Phi^{**}.$ In particular  $\Phi^{**}(1)=\Phi^{**}(1)\Psi^{**}(1)\Phi^{**}(1).$\smallskip

$(2)\Rightarrow (1)$ Let $v\in B^{**}$ such that $\Phi^{**}(1)=\Phi^{**}(1)v\Phi^{**}(1).$
The mapping $\Psi'=L_v\circ R_v \circ\phi:A\to B^{**}$ is continuous, and by Lemma \ref{l non-unital with projections automatic continuity} $(b),$ we have
 $$\Phi (b) = \Phi^{**} (1) \Psi(b) \Phi^{**}(1),\quad\forall\;b\in A ,$$ and consequently $$\Phi (b) v \Phi^{**} (1) = \Phi^{**} (1) \Psi(b) \Phi^{**}(1) v  \Phi^{**} (1) = \Phi (b),$$ and $$\Phi^{**} (1) v \Phi (b) = \Phi^{**} (1) v \Phi^{**} (1) \Psi(b) \Phi^{**}(1) = \Phi (b) ,\quad\forall\;b\in A$$
Now, for arbitrary $a,\;b\in A,$ we get:
 $$\Phi (aba) = \Phi(a)\Psi(b)\Phi(a)=\Phi(a)v\Phi^{**} (1)\Psi(b)\Phi^{**} (1)v\Phi(a)$$ $$=\Phi(a)v\Phi(b)v\Phi(a)=\Phi(a)\Psi'(b)\Phi(a)$$
 and
 $$\Psi' (aba) = v\Phi(aba)v=v\Phi(a)\Psi(b)\Phi(a)v$$
 $$=v\Phi(a)v\Phi^{**} (1)\Psi(b)\Phi^{**} (1)v\Phi(a)v=\Psi'(a)\Phi(b)\Psi'(a).$$
\end{proof}

We can now extend our study to linear maps between $K(H)$ spaces.

\begin{theorem}\label{p automatic continuity on K(H)} Let $\Phi, \Psi: K(H_1) \to K(H_2)$ be linear maps such that $\Phi$ is continuous and $(\Phi,\Psi)$ is Jordan-triple multiplicative. Then $\Phi$ admits a continuous normalized-pg-inverse.
\end{theorem}

\begin{proof} We may assume that $H_1$ is infinite dimensional.\smallskip

We shall first prove that for every infinite family $\{p_j :j \in \Lambda\}$ of mutually orthogonal projections in $K(H_1)$ the set \begin{equation}\label{eq boundedness on mo projections} \{\Psi (p_j) : j \in \Lambda\} \hbox{ is bounded.}
\end{equation} Arguing by contradiction, we assume that the above set is unbounded. Then we can find a countable subset $\Lambda_0$ in $\Lambda$ such that $\|\Psi (p_n)\| \geq n^3,$ for every natural $n$.  Since the projections in the sequence $(p_n)$ are mutually orthogonal, the element $\displaystyle x_0 = \sum_{k=1}^{\infty} \frac1n p_n\in K(H_1),$ and by hypothesis,
$$\Psi(x_0) \Phi(p_n) \Psi(x_0) = \Psi(x_0 p_n x_0) = \frac{1}{n^2} \Psi(p_n),$$ and hence $$n= \frac{1}{n^2} n^3 < \left\| \frac{1}{n^2} \Psi(p_n) \right\|\leq \left\| \Psi(x_0) \right\|^2 \left\|\Phi(p_n) \right\| \leq \left\| \Psi(x_0) \right\|^2 \left\|\Phi\right\|,$$ for every natural $n$, which is impossible.\smallskip

Now, let $\{p_j :j \in \Lambda\}$ be a maximal set of mutually orthogonal (minimal) projections in $K(H_1)$. By \eqref{eq boundedness on mo projections} there exists a positive $R$ such that $\|\Psi (p_j)\|\leq R$, for every $j\in \Lambda$. Let $\mathcal{F}(\Lambda)$ denote the collection of all finite subsets of $\Lambda$, ordered by inclusion. For each $F\in \mathcal{F}(\Lambda)$ we set $\displaystyle q_{_F} :=\sum_{j\in F} p_j\in K(H_1).$ It is known that $(q_{_F})_{_{F\in \mathcal{F}(\Lambda)}})$ is an approximate unit in $K(H_1)$. Clearly for each $F\in \mathcal{F}(\Lambda)$ we have $\|\Psi (q_{_F})\|\leq (\sharp F) \; R.$\smallskip

We shall now prove that
\begin{equation}\label{eq boundedness on the approx unit} \{\Psi (q_{_F}) : F\in \mathcal{F}(\Lambda) \} \hbox{ is bounded.}
\end{equation} Suppose, contrary to our goal, that the above set is unbounded.\smallskip

Now, we shall establish the following property: for each $F\in \mathcal{F}(\Lambda)$, and each positive $\delta$ there exists $G\in \mathcal{F}(\Lambda)$ with $G\cap F = \emptyset$ and $\| \Psi (q_{_G}) \| > \delta.$ Indeed, if that is not the case, there would exist $F \in \mathcal{F}(\Lambda)$, $\delta>0$ such that $\| \Psi (q_{_G}) \| \leq  \delta, $  for every $G\in \mathcal{F}(\Lambda)$ with $G\cap F = \emptyset$. In such a case, for each $H\in \mathcal{F}(\Lambda)$ we have $$\|\Psi (q_{_H}) \| \leq  \|\Psi (q_{_{(H\cap F)}}) \| + \|\Psi (q_{_{(H\cap F^{c})}}) \| \leq (\sharp F) \; R + \delta,$$ which contradicts the unboundedness of the set $\{\Psi (q_{_F}) : F\in \mathcal{F}(\Lambda) \}$.\smallskip

Applying the above property, we find a sequence $(F_n)\subset \mathcal{F}(\Lambda)$ with $F_n \cap F_m =\emptyset $ for every $n\neq m$ and $\|\Psi (q_{_{F_n}})\| > n^3,$ for every natural $n$. We take $\displaystyle y_0 :=\sum_{n=1}^{\infty} \frac1n q_{_{F_n}}\in K(H_1)$.
By hypothesis, $\Psi(y_0) \Phi(q_{_{F_n}}) \Psi (y_0) = \Psi (y_0 q_{_{F_n}} y_0) = \frac{1}{n^2} \Psi (q_{_{F_n}}),$ and hence
$$n=\frac{1}{n^2} n^3 < \|\Psi(y_0) \Phi(q_{_{F_n}}) \Psi (y_0)  \|\leq \|\Psi(y_0)\|^2 \|\Phi\|,$$ for every natural $n$, leading to the desired contradiction. This concludes the proof of \eqref{eq boundedness on the approx unit}.\smallskip

Now, by \eqref{eq boundedness on the approx unit} the net $(\Psi (q_{_F}))_{_{ F\in \mathcal{F}(\Lambda)}}$ is bounded in $K(H_2)\subseteq B(H_2),$ and by the weak$^*$-compactness of the closed unit ball of the latter space, we can find a subnet $(\Psi (q_{j}))_{j\in \Lambda'}$ converging to some $w\in B(H_2)$ in the weak$^*$ topology of this space. We observe that $(q_{_F})_{_{ F\in \mathcal{F}(\Lambda)}}\to 1$ in the weak$^*$ topology of $B(H_1)$, and by the weak$^*$ continuity of $\Phi^{**}$ we also have $(\Phi (q_{j}))_{j\in \Lambda'}\to \Phi^{**} (1)$ in the weak$^*$ topology of $B(H_2)$. Lemma \ref{l non-unital with projections automatic continuity} implies that $$  \Phi (q_{j}) = \Phi^{**} (1) \Psi (q_{j}) \Phi^{**} (1)$$ for every $j\in \Lambda'$. Taking weak$^*$ limits in the above equality we get $$\Phi^{**} (1) = \Phi^{**} (1) w \Phi^{**} (1),$$ and hence $\Phi^{**} (1)$ is regular in $B(H_2)$.\smallskip

Finally, an application of Lemma \ref{l characterization of automatic continuity of a npg-inverse with Phi**(1) regular} gives the desired statement.
\end{proof}

We can now obtain an improved version of Corollary \ref{c prop2 continuity instead unital} for linear maps between $K(H)$ spaces.

\begin{corollary}\label{c prop2 continuity instead unital improved}
Let $\Phi,\Upsilon:K(H_1)\rightarrow K(H_2)$ be linear maps such that $\Phi$ is continuous and $(\Phi,\Upsilon)$ is Jordan-triple multiplicative. Then the following statements hold:
\begin{enumerate}[$(a)$]
    \item There exists a continuous Jordan homomorphism  $T:K(H_1)\rightarrow B(H_2)$ such that $\Phi (a) = T (a) \Phi^{**} (1)$, for every $a\in K(H_1)$, and $\Phi^{**}(1) B(H_2) =T(1) B(H_2);$
    \item There exists a continuous Jordan homomorphism $S:K(H_1)\rightarrow B(H_2)$ such that $\Phi (a)=\Phi^{**} (1) S (a),$ for every $a\in K(H_1)$, and $B(H_2) \Phi^{**}(1)=B(H_2) S(1).$
\end{enumerate}
\end{corollary}

\begin{proof} By Theorem \ref{p automatic continuity on K(H)} $\Phi$ admits a continuous normalized-pg-inverse $\Psi : K(H_1)\to B(H_2)$. Applying Lemma \ref{l non-unital with projections automatic continuity} we deduce that the mappings $T, S :  K(H_1)\to B(H_2)$, $T(a) = \Phi^{**} (1) \Psi (a)$ and $S(a) = \Psi(a) \Phi^{**} (1)$ ($a\in K(H_1)$), are linear and continuous and the identities $$ T(a) T(b) = \Phi (a) \Psi (b) , \ \Phi (a) = T(a) \Phi^{**} (1),$$ and $$ S(a) S(b) = \Psi (a) \Phi (b) , \ \Phi (a) = \Phi^{**} (1)  T(a),$$ hold for every $a,b\in K(H_1)$.\smallskip

Let $(u_{\lambda})$ be an approximate unit in $K(H_1)$. Applying the separate weak$^*$ continuity of the product of $B(H_2)$ we have $$\Psi (a) \Phi^{**} (1) \Psi (a) = \hbox{weak$^*$-}\lim_{\lambda} \Psi(a) \Phi(u_\lambda) \Psi (a) $$ $$= \hbox{weak$^*$-}\lim_{\lambda}   \Psi(a u_\lambda a) = \Psi^{**} (a^2) = \Psi (a^2),$$ for all $a\in K(H_1)$. Finally, by Lemma \ref{l non-unital with projections automatic continuity} we get $$T(a)^2 = \Phi^{**} (1) \Psi (a) \Phi^{**} (1) \Psi (a) = \Phi^{**} (1) \Psi (a^2) = T(a^2),$$ for all $a$ in $K(H_1)$. The statement for $S$ follows by similar arguments.
\end{proof}

Let $\Phi: K(H_1) \to K(H_2)$ be a bounded linear map. We do not know if any normalized-pg-inverse of $\Phi$ is automatically continuous.

\section{Pointwise-generalized-inverses of linear maps between JB$^*$-triples}

In this section we explore a version of pointwise-generalized inverse in the setting of JB$^*$-triples.

\begin{definition}\label{def pg-inverse triples} Let $\Phi:E\to F$ be a linear mapping between JB$^*$-triples. We shall say that $T$ admits a pointwise-generalized-inverse (pg-inverse) if there exists a linear mapping $\Psi: E\to F$ satisfying $$\Phi\{a,b,c\} = \{\Phi(a), \Psi(b), \Phi(c)\},$$ for every $a,b,c\in E$. If $\Phi$ also is a pg-inverse of $\Psi$ we shall say that $\Psi$ is a normalized-pg-inverse of $\Phi$ or that $(\Phi,\Psi)$ is JB$^*$-triple multiplicative.
\end{definition}

Let $\Phi,\Psi: A\to B$ be linear maps between C$^*$-algebras. The pair $(\Phi,\Psi)$ is Jordan-triple multiplicative if $\Phi (aba ) = \Phi(a)\Psi(b)\Phi(a)$ and $\Psi (aba ) = \Psi(a)\Phi(b)\Psi(a)$. C$^*$-algebras can be regarded as JB$^*$-triples and in such a case, the couple $(\Phi,\Psi)$ is JB$^*$-triple multiplicative if $\Phi (a b^* a ) = \Phi(a)\Psi(b)^*\Phi(a)$ and $\Psi (ab^*a ) = \Psi(a)\Phi(b)^*\Psi(a)$. We should remark, that these two notions are, in principle, independent.\smallskip

Every triple homomorphism between JB$^*$-triples is a normalized-pg-inverse of itself. The next lemma gathers some basic properties of linear maps between JB$^*$-triples admitting a pg-inverse.

\begin{lemma}\label{l basic pg invertible triples}  Let $\Phi:E\rightarrow F$ be a linear map between JB$^*$-triples admitting a
pg-inverse $\Psi$. Then the following statements hold:
\begin{enumerate}[$(a)$]
\item $\Phi$ maps von Neumann regular elements in $E$ to von Neumann regular elements in $F$, that is, $\Phi$ is a weak regular preserver, More concretely, if $b$ is a generalized inverse of $a$ then $\Psi(b)$ is a generalized inverse of $\Phi(a)$;
\item Let $\Phi_1 : A\to E$ and $\Phi_2: F\to B$ be linear maps between JB$^*$-triples admitting a pg-inverse, then $\Phi_2 \Phi$ and $\Phi\Phi_1$ admit a pg-inverse too;
\item If $\Phi$ and $\Psi$ are continuous then $\Psi^{**}: E^{**} \to F^{**}$ is a pg-inverse of $\Phi^{**}$.
\end{enumerate}
\end{lemma}

\begin{proof}$(a)$ If $a$ is von Neumann regular the there exists $b\in E$ such that $Q(a) (b)= \{a,b,a\} = a$. By hypothesis, $\Phi(a) = \Phi\{a,b,a\} =\{\Phi(a), \Psi(b), \Phi(a)\}$, which shows that $\Phi(a)$ is von Neumann regular.\smallskip

$(b)$ Under these hypothesis, let $\Psi_1$ be a pg-inverse of $\Phi_1$. Then $$\Phi_1 \Phi \{a,b,a\} = \Phi_1 \{\Phi(a), \Psi(b), \Phi(a)\}= \{\Phi_1 \Phi(a), \Psi_1 \Psi(b), \Phi_1\Phi(a)\},$$ which shows that $\Psi_1 \Psi$ is a pg-inverse of $\Phi_1\Phi$. The rest of the statement follows from similar arguments.\smallskip

$(c)$ Assuming that $\Phi$ and $\Psi$ are continuous, the maps $\Phi^{**}$, $\Psi^{**}$ are weak$^*$-continuous. The bidual $E^{**}$ is a JBW$^*$-triple, and hence its triple product is separately weak$^*$ (see \cite{BarTi}). Then we can repeat the arguments in the proof of Lemma \ref{l bitransposed} to conclude, via Goldstine's theorem, that $$\Phi^{**}\{a,b,c\} = \{\Phi^{**}(a), \Psi^{**}(b), \Phi^{**}(c)\},$$ for every $a,b,c\in E^{**}$.
\end{proof}

Let us observe that the arguments in the proof of Theorem \ref{thm2} are obtained with geometric tools which are not merely restricted to the setting of C$^*$-algebras. Our next result is a generalization of the just commented theorem, to clarify the parallelism, we recall that, by Kadison's theorem (\cite[Proposition 1.6.1 and Theorem 1.6.4]{S}), a C$^*$-algebra $A$ is unital if and only if its closed unit ball contains extreme points.

\begin{theorem}\label{thm2 triples} Let $\Phi, \Psi:E\rightarrow F$ be linear maps between JB$^*$-triples. Suppose that $(\Phi,\Psi)$ is JB$^*$-triple multiplicative. Then the following are equivalent:
\begin{enumerate}[$(a)$]
\item $\Phi$ and $\Psi$ are contractive;
\item  $\Psi=\Phi$ is a triple homomorphism.
\end{enumerate}
If the closed unit ball of $E$ contains extreme points, then the above statements are also equivalent to the following:
\begin{enumerate}[$(c)$]\item $\Phi$ \emph{strongly preserves regularity}, that is, $\Phi(x^{\wedge}) = \Phi(x)^{\wedge}$ for every $x\in E^{\wedge}$.
\end{enumerate}
\end{theorem}

\begin{proof}$(a)\Rightarrow (b)$ By Lemma \ref{l basic pg invertible triples}$(c)$, $\Psi^{**}$ is a normalized-pg-inverse of $\Phi^{**}$. Let $e$ be  a tripotent in $E^{**}$. The maps $\Psi^{**}$ and $\Phi^{**}$ are contractive, and by Lemma \ref{l basic pg invertible triples}$(a)$, $\Psi^{**}(e)$ is a generalized inverse of $\Phi^{**}(e)$ and both lie in the closed unit ball of $F^{**}$. Corollary 3.6 in \cite{BurKaMoPeRa08} assures that $\Phi^{**}(e)$ and $\Psi^{**}(e)$ both are tripotents in $F^{**}$. Let us assume that $\Phi^{**}(e)$ (equivalently, $\Psi^{**}(e)$) is non-zero. The identity \begin{equation}\label{eq new thm triples} \Phi^{**}(e) = \{\Phi^{**}(e), \Psi^{**}(e),\Phi^{**}(e) \}
 \end{equation} implies that $P_2(\Phi^{**}(e)) (\Psi^{**}(e)) = \Phi^{**}(e)$. Lemma 1.6 in \cite{FriRu85} assures that $$\Psi^{**}(e) = \Phi^{**}(e) + P_0(\Phi^{**}(e)) (\Psi^{**}(e))$$ and similarly $$\Phi^{**}(e) = \Psi^{**}(e) + P_0(\Psi^{**}(e)) (\Phi^{**}(e)).$$ We deduce from \eqref{eq new thm triples} that $\Phi^{**}(e) = \Psi^{**}(e)$, for every tripotent $e\in E^{**}$.\smallskip

In a JBW$^*$-triple every element can be approximated in norm by a finite linear combination of mutually orthogonal tripotents (see \cite[Lemma 3.11]{Horn87}). We can therefore guarantee that $\Phi^{**}= \Psi^{**}$ is a triple homomorphism.\smallskip

The implication $(b)\Rightarrow (a)$ is established in \cite[Lemma 1$(a)$]{BarDanHor88}.\smallskip

The final statement follows from \cite[Theorem 3.2]{BurMarMoPe2016}.
\end{proof}

The next corollary, which is an extension of Corollary \ref{c Jordan homomorphisms} for JB$^*$-algebras, is probably part of the folklore in JB$^*$-algebra theory but we do not know an explicit reference.

\begin{corollary}\label{c contractive are symetric in JB*-algebras} Let $A$ and $B$ be JB$^*$-algebras and let $\Phi:A\rightarrow B$ be a Jordan homomorphism.  Then the following statements are equivalent:\begin{enumerate}[$(a)$] \item $\Phi $ is a contraction;
\item $\Phi$ is a symmetric map {\rm(}i.e. $\Phi$ is a Jordan $^*$-homomorphism{\rm)};
\item $\Phi$ is a triple homomorphism.
\end{enumerate}
If the closed unit ball of $A$ contains extreme points, then the above statements are also equivalent to the following:
\begin{enumerate}[$(c)$]\item $\Phi$ \emph{strongly preserves regularity}, that is, $\Phi(x^{\wedge}) = \Phi(x)^{\wedge}$ for every $x\in A^{\wedge}$.
\end{enumerate}
\end{corollary}

\begin{proof} In the hypothesis of the Corollary, we observe that the identities $$\Phi \{a,b,a\} = \Phi (U_{a} (b^*)) = U_{\Phi(a)} (\Phi(b^*)) = \{\Phi (a), \Phi(b^*)^*, \Phi (a)\},$$ $$\Phi( \{a,b,a\}^*)^* = \Phi (U_{a^*} (b))^* = U_{\Phi(a^*)^*} (\Phi(b)^*) = \{\Phi (a^*)^*, \Phi(b), \Phi (a^*)^*\},$$ hold for every $a,b\in A$. This shows that the mapping $x\mapsto  \Psi(x) = \Phi(x^*)^*$ is a normalized-pg-inverse of $\Phi$. \smallskip

$(a)\Rightarrow (b)$ If $\Phi$ is contractive then $\Psi$ is contractive too, and it follows from Theorem \ref{thm2 triples} that $\Psi = \Phi$, or equivalently, $\Phi (a^*) = \Phi (a)^*$ for every $a$. The other implications have been proved in Theorem \ref{thm2 triples}.
\end{proof}

Returning to Corollaries \ref{c Jordan homomorphisms} and \ref{c contractive are symetric in JB*-algebras}, in a personal communication, M. Cabrera and A. Rodr{\'\i}guez noticed that, though an explicit reference for these results seems to be unknown, they can be also rediscovered with arguments contained in their recent monograph \cite{CaRod14}. We thank Cabrera and Rodr{\'i}guez for bringing our attention to the lemma and arguments presented below, and for providing the appropriate connections with the results in \cite{CaRod14}.

\begin{lemma}\label{211013dd}
Let $A$ be a $JB^*$-algebra, and let $e$ be an idempotent in $A$ such that $ \| e \|  =1$. Then $e^*=e$.
\end{lemma}

\begin{proof} By \cite[Proposition 3.4.6]{CaRod14}, the closed subalgebra of $A$
generated by $\{e,e^*\}$ is a JC$^*$-algebra (i.e. a norm closed Jordan $^*$-subalgebra of a C$^*$-algebra).
Therefore $e$ can be regarded as a norm-one idempotent in a C$^*$-algebra, so that, by
\cite[Corollary 1.2.50]{CaRod14}, we have $e^*=e$, as required.
\end{proof}

The unital version of Corollary \ref{c contractive are symetric in JB*-algebras} is treated in \cite[Corollary
3.3.17(a)]{CaRod14}. The general statement needs a more elaborated argument to rediscover Corollary \ref{c contractive are symetric in JB*-algebras}.


\begin{proof}[New proof of Corollary \ref{c contractive are symetric in JB*-algebras}] Let $\Phi:A\to B$ be a contractive Jordan homomorphism between JB$^*$-algebras. If $A$ and $B$ are unital and $\Phi$ maps the unit in $A$ to the unit in $B$, then the result follows from  \cite[Corollary 3.3.17(a)]{CaRod14}.\smallskip

We deal now with the general statement. We may assume that $\Phi \neq
0$. It is known that $A^{**}$ and $B^{**}$ are unital JB$^*$-algebras
whose products and involutions extend those of $A$ and $B$, respectively
(cf. \cite[Proposition 3.5.26]{CaRod14}), $\Phi^{**} :A^{**}\to
B^{**}$ is a contractive Jordan algebra homomorphism (cf. \cite[Lemma 3.1.17]{CaRod14}),
and $e:=\Phi (1)$ is a norm-one idempotent in $B^{**}$. Therefore, by Lemma \ref{211013dd} and \cite[Lemma 2.5.3]{CaRod14},
$U_e (B^{**})$ is a closed Jordan $*$-subalgebra of $B^{**}$ (hence a unital
JB$^*$-algebra) containing $\Phi^{**}(A^{**})$. Then $\Phi^{**}$,
regarded as a mapping from $A^{**}$ to  $U_e(B^{**})$, becomes a
unit-preserving contractive algebra homomorphism. By the first
paragraph of this proof, $\Phi^{**}$ (and hence $\Phi $) is a
$*$-mapping.
\end{proof}

\end{document}